\documentclass[12pt]{amsart}

\usepackage{amsmath,amsthm, amssymb, amsfonts, epsfig,amscd}
\usepackage[all]{xy}

\usepackage[mathscr]{eucal}
\usepackage{mathrsfs}
\usepackage{latexsym}
\usepackage{graphicx}
\usepackage[all]{xy}
\usepackage[table]{xcolor}

\theoremstyle{plain}
\newtheorem{thm}[subsection]{Theorem}
\newtheorem{lem}[subsection]{Lemma}
\newtheorem{prop}[subsection]{Proposition}
\newtheorem{cor}[subsection]{Corollary}
\newtheorem{lemma}[subsection]{Lemma}

\theoremstyle{definition}

\newtheorem{remark}[subsection]{Remark}

\def\vs{\vskip}
\def\ni{\noindent}

\begin{document}
\title[Cotangent bundle to the Grassman variety]%
{Cotangent bundle to the Grassman variety}
\author[V. Lakshmibai]{V. Lakshmibai${}^{\dag}$}
\address{Northeastern University, Boston, USA}
   \email{lakshmibai@neu.edu}
 
\thanks{${}^{\dag}$ V.~Lakshmibai was partially supported by
NSA grant H98230-11-1-0197,  NSF grant 0652386}
%{Northeastern University RSDF 07--08.}

 \maketitle

\begin{abstract}%
%\vskip-1cm
We show that there is  an affine Schubert variety in the infinite
dimensional partial Flag variety (associated to the two-step
parabolic subgroup of the Kac-Moody group ${\widehat{SL_n}}$,
corresponding to omitting $\alpha_0,\alpha_d$) which is a natural
compactification of the cotangent bundle to the Grassmann variety.
%\vskip.2cm
\end{abstract}
\section{Introduction} Let the base field $K$ be the field of
complex numbers. Consider a cyclic quiver with $h$ vertices and
dimension vector ${\underline d}=(d_1,\cdots, d_h)$:
$$\xymatrix{ 1 \ar[r] & 2 \ar[r] & \cdots & h-2 \ar[r] &
h-1\ar[dll]\\ &&h \ar[ull] & & }$$ Denote $V_i=K^{d_i}$. Let
$$Z=Hom(V_1,V_{2})\times \cdots\times Hom(V_h,V_{1}),\
GL_{{\underline d}}=\prod_{1\le i\le h}\,GL{(V_i)}$$ We have a
natural action of $GL_{{\underline d}}$ on $Z$: for
$g=(g_1,\cdots,g_h)\in GL_{{\underline d}},f=(f_1,\cdots,f_h)\in
Z$,
$$g\cdot f=(g_2f_1g_1^{-1},g_3f_2g_2^{-1},\cdots,g_1f_hg_h^{-1})$$
Let $$\mathcal{N}=\{(f_1,\cdots,f_h)\in Z\,|\,f_h\circ
f_{h-1}\circ\cdots\circ f_1:V_1\rightarrow V_1{\mathrm{\ is\
nilpotent}}\}$$

Note that $f_h\circ f_{h-1}\circ\cdots\circ f_1:V_1\rightarrow
V_1$ is equivalent to

\ni $f_{i-1}\circ f_{i-2}\circ\cdots\circ f_1\circ f_h
\circ\cdots\circ f_{i+1} f_i:V_i\rightarrow V_i$ is nilpotent.
Clearly $\mathcal{N}$ is $GL_{{\underline d}}$-stable. Lusztig
(cf.\cite{lus}) has shown that an orbit closure in $\mathcal{N}$
is canonically isomorphic to an open subset of a Schubert variety
in ${\widehat{SL}}_n/Q$, where $n=\sum_{1\le i\le h}\, d_i$, and
$Q$ is the parabolic subgroup of ${\widehat{SL}}_n$ corresponding
to omitting $\alpha_0,\alpha_{d_{1}},
\alpha_{d_{1}+d_{2}},\cdots,\alpha_{d_{1}+\cdots +d_{h-1}}$
($\alpha_i, 0\le i\le n-1$ being the set of simple roots for
${\widehat{SL}}_n$). Corresponding to $h=1$, we have that
$\mathcal{N}$ is in fact the variety of nilpotent elements in
$M_{d_{1},d_{1}}$, and thus the above isomorphism identifies
$\mathcal{N}$ with an open subset of a Schubert variety
$X_{\mathcal{N}}$ in ${\widehat{SL}}_n/\mathcal{P},\mathcal{P}$,
being the maximal parabolic subgroup of ${\widehat{SL}}_n$
corresponding to ``omitting" $\alpha_0$.

Let now $h=2$, and let
$$Z_0=\{(f_1,f_2)\in Z\,|\,f_2\circ f_1=0,f_1\circ f_2=0\}$$
Strickland (cf. \cite{st}) has shown that each irreducible
component of $Z_0$ is the conormal variety to a determinantal
variety in $M_{d_1,d_2}$. A determinantal variety in $M_{d_1,d_2}$
being canonically isomorphic to an open subset in a certain
Schbert variety in $G_{d_2,d_1+d_2}$ (the Grassmannian variety of
$d_2$-dimensional subspaces of $K^{d_1+d_2}$) (cf.\cite{g/p-2}),
the above two results of Lusztig and Strickland suggest a
connection between conormal varieties to Schubert varieties in the
(finite-dimensional) flag variety and the affine Schubert
varieties. This is the motivation for this article. We define a
canonical embedding of $T^*G_{d,n}$ ($G_{d,n}$ being the Grassmann
variety of $d$-dimensional subspaces of $K^n$) inside a
$\mathcal{P}$-stable affine Schubert variety $X(\kappa_d)$ inside
${\widehat{SL}}_n/\mathcal{Q}$ ($\mathcal{Q}$ being the two-step
parabolic subgroup of ${\widehat{SL_n}}$, corresponding to
omitting $\alpha_0,\alpha_d$), and show that $X(\kappa_d)$ gives a
natural compactification of $T^*G_{d,n}$ (cf. Theorem \ref{ctgt}).

The fact that the affine Schubert variety $X(\kappa_d)$ is a natural
compactification of $T^*G_{d,n}$, suggests similar compactifications
for conormal varieties to Schubert varieties in $G_{d,n}$ (by
suitable affine Schubert varieties in
${\widehat{SL}}_n/\mathcal{Q}, \mathcal{Q}$ being as above). The details will appear in a subsequent paper.

The sections are organized as follows. In \S \ref{prelim}, we fix
notation and recall \emph{affine Schubert varieties}. In \S
\ref{elt}, we introduce the element $\kappa_d$ (in
${\widehat{W}}$, the affine Weyl group), and prove some properties
of $\kappa_d$. In \S \ref{main}, we prove one crucial result on
$\kappa_d$ needed for proving the embedding of $T^*G_{d,n}$ inside
${\widehat{SL}}_n/\mathcal{Q}$. In \S \ref{bundle}, we present the
results for $T^*G_{d,n}$.

\vs.2cm\ni\textbf{Acknowledgement:}  The author expresses her thanks to Littelmann and Seshadri for some useful discussion.
%%%%%%%%%%%%%%%%%%%%%%%%%%%%%%%%%%%%%%%%%%%%%%%%%%%%%%%%%%%%
\section{Affine Schubert varieties}\label{prelim} Let $K=\mathbb{C}, F=K((t))$,
the field of Laurent series, $A^{\pm}=K[[t^{\pm 1}]]$. Let $G$ be a
semi-simple algebraic group over $K$, $T$ a maximal torus in $G$,
$B$, a Borel subgroup, $B\supset T$ and let $B^-$ be the Borel
subgroup opposite to $B$. Let $\mathcal{G}=G(F)$. The natural
inclusions
$$K\hookrightarrow A^{\pm}\hookrightarrow F$$ induce inclusions
$$G\hookrightarrow G(A^{\pm})\hookrightarrow\mathcal{G}$$ The
natural projections
$$A^{\pm}\rightarrow K,t^{\pm 1}\mapsto 0$$
induce homomorphisms
$$\pi^{\pm}:G(A^{\pm})\rightarrow G$$ Let
$$\mathcal{B}=(\pi^+)^{-1}(B),\mathcal{B}^-=(\pi^-)^{-1}(B^-)$$ Let
${\widehat W}=N(K[t,t^{-1}])/T$, the \emph{affine Weyl group} of
$G$ (here, $N$ is the normalizer of $T$ in $G$); ${\widehat W}$ is
a Coxeter group on $\ell+1$ generators $\{s_0,s_1\cdots s_\ell\}$,
where $\ell$ is the rank of $G$, and $\{s_0,s_1\cdots s_\ell\}$ is
the set of reflections with respect to the simple roots,
$\{\alpha_0,\alpha_1\cdots\alpha_\ell\}$, of $\mathcal{G}$.

\vs.2cm\ni\textbf{Bruhat decomposition:} We have
$$G(F)=\dot\cup_{w\in{\widehat W}}\mathcal{B}w\mathcal{B},
G(F)/\mathcal{B}=\dot\cup_{w\in{\widehat
W}}\mathcal{B}w\mathcal{B}(mod\,\mathcal{B})$$ For $w\in{\widehat
W}$, let $X(w)$ be the \emph{affine Schubert variety} in
$G(F)/\mathcal{B}$:
$$X(w)=\dot\cup_{\tau\le w}\mathcal{B}\tau\mathcal{B}(mod\,
\mathcal{B})$$ It is a projective variety of dimension $\ell(w)$.

\subsection{Affine Flag variety,  Affine Grassmannian:}

Let $G=SL(n)$,

\ni $\mathcal{G}=G(F), G_0=G(A^+)$. Then $\mathcal{G}/\mathcal{B}$
is the \emph{affine Flag variety}, and $\mathcal{G}/G_0$ is the
\emph{affine Grassmannian}. Further,
$$\mathcal{G}/G_0=\dot\cup_{w\in{\widehat W}^{G_0}}
\mathcal{B}w\,G_0 (mod\,G_0)$$ where ${\widehat W}^{G_0}$ is the
set of minimal representatives in ${\widehat W}$ of ${\widehat
W}/W_{G_0}$.

\ni Denote $A^{+}(=k[[t]])$ by just $A$. Let $${\widehat
{Gr(n)}}=\{A\texttt{-lattices\ in\ } F^n \}$$ Here, by an
$A$-lattice in $F^n$, we mean a free $A$-submodule of $F^n$ of
rank $n$. Let $E$ be the standard lattice, namely, the $A$-span of
the standard $F$-basis $\{e_1,\cdots,e_n\}$ for $F^n$. For $V\in
{\widehat {Gr(n)}}$, define
$$\mathrm{vdim}(V):=dim_K(V/V\cap E)-dim_K(E/V\cap E)$$
 One refers to $\mathrm{vdim}(V)$ as the
 \emph{virtual dimension of $V$}. For $j\in \mathbb{Z}$ denote
$${\widehat {Gr_j(n)}}=\{V\in{\widehat {Gr(n)}}\,|\,
\mathrm{vdim}(V)=j\}$$ Then ${\widehat {Gr_j(n)}},j\in\mathbb{Z}$
give the connected components of ${\widehat {Gr(n)}}$. We have a
transitive action of $GL_n(F)$ on ${\widehat {Gr(n)}}$ with
$GL_n(A)$ as the stabilizer of the standard lattice $E$. Further,
let $\mathcal{G}_0$ be the subgroup of $GL_n(F)$, defined as,
$$\mathcal{G}_0=\{g\in GL_n(F)\,|\,\mathrm{ord(det\, }g)=0\}$$
(here, for a $f\in F$, say $f=\sum\,a_it^i$, \emph{order f} is the
smallest $r$ such that $a_r\ne 0$). Then $\mathcal{G}_0$ acts
transitively on ${\widehat {Gr_0(n)}}$ with $GL_n(A)$ as the
stabilizer of the standard lattice $E$. Also, we have a transitive
action of $SL_n(F)$ on ${\widehat {Gr_0(n)}}$ with $SL_n(A)$ as
the stabilizer of the standard lattice $E$. Thus we obtain the
identifications:
$$\begin{gathered}GL_n(F)/GL_n(A)\simeq{\widehat {Gr(n)}}\\
\mathcal{G}_0/GL_n(A)\simeq{\widehat {Gr_0(n)}},
SL_n(F)/SL_n(A)\simeq{\widehat {Gr_0(n)}}
\end{gathered}\leqno{(*)}$$ In particular, we obtain
$$\mathcal{G}_0/GL_n(A)\simeq SL_n(F)/SL_n(A)\leqno{(**)}$$

\subsection{Generators for ${\widehat W}$:}\label{gen} Following
the notation in \cite{kac}, we shall work with the set of
generators for ${\widehat W}$ given by $\{s_0, s_1,\cdots,
s_{n-1}\}$, where $s_i, 0\le i\le n-1$ are the reflections with
respect to $\alpha_i, 0\le i\le n-1$. Note that $\{\alpha_i, 1\le
i\le n-1\}$ is simply the set of simple roots of $SL_n$ (with
respect to the Borel subgroup $B$); also note that
$\alpha_0=\delta-\theta$, where $\theta$ is the highest root of
the (finite) Type $\mathbf{A}_{n-1}$ with simple roots
$\{\alpha_1,\cdots\alpha_{n-1}\}$:
$$\theta=\alpha_1+\cdots+\alpha_{n-1}$$ We have the following canonical lifts (in $\mathcal{G}$) for $s_i,
0\le i\le n-1$; for $1\le i\le n-1$, $s_i$ is the permutation
matrix $(a_{rs})$, with $a_{jj}=1,j\ne i,i+1,\  a_{i\,i+1}=1,
a_{i+1\,i}=-1$, and all other entries are $0$. A lift for $s_0$ is
given by
$$\begin{pmatrix}
0&0&\cdots & t^{-1}\\
0&1&\cdots &0\\
\vdots & \vdots & \vdots & \vdots\\
0&\cdots &1 &0\\
-t &0&0&0
\end{pmatrix}$$
%%%%%%%%%%%%%%%%%%%%%%%%%%%%%%%%%%%%%%%%%%%%%%%%%%%%%%%%%%
\section{The element $\kappa_d$}\label{elt} 

Let $P$ be a parabolic subgroup, and $W_P$ the Weyl group of $P$. Let $R_P$ be the set of roots of $P$, and $S_P$ the set of simple roots of $P$, The Schbert varieties in $G/P$ are indexed by $W/W_P$. We gather some well-known facts on $W^P$, the set of minimal representatives in $W$ of the elements of $W/W_P$.  

For $wP\in W/W_P$, there exists a (unique) representative $w_{min}\in W$ with the following properties:

\ni{\textbf {Fact 1:}} Among all the representatives in $W$ for $wP$, $w_{min}$ is the unique element of smallest length.

\ni{\textbf {Fact 2:}} $w_{min}(\alpha)>0, \forall \alpha\in S_P$.

\ni{\textbf {Fact 3:}} dim$\,X(w)=l(w_{min})$ ($X(w)$ is the Schubert variety in $G/P$, corresponding to $w$).

\ni{\textbf {Fact 4:}}  Let  $P_\alpha$ be the parabolic subgroup with $\{\alpha\}$ as the associated set of simple roots. Then $X(w)$ is stable for left multiplication by $P_\alpha$ if and only if $s_\alpha w<w\, (mod\, W_P)$. More generally, given a parabolic sub group $Q$,  $X(w)$ is stable for left multiplication by $Q$ if and only if $s_\alpha w<w\, (mod \,W_P), \forall\alpha\in S_Q$.

\begin{remark} 
These results hold for Kac-Moody groups also.
\end{remark}

 Denote the Weyl group of $SL(n)$ by $W$
(note that $W$ is just the symmetric group $S_n$). Consider the
Type $\mathbf{A}_{n-1}$ Dynkin diagram with simple roots
$\alpha_1,\cdots,\alpha_{n-1}$, namely
$$\xymatrix{ \alpha_1 \ar [r]  & \cdots \ar [r] & \alpha_{n-1}}\leqno{(A)}$$
 Let
$W_{P_d}$ be the Weyl group of $P_d$, the maximal parabolic
subgroup of $SL(n)$ corresponding to omitting the simple root
$\alpha_d$. Note that $P_d$ consists of $\{(a_{ij})\in SL(n)\}$
such that $a_{ij},j\le d<i\le n$ are $0$. Note also that
$W_{P_d}=S_d\times S_{n-d}$. We may suppose $d\le n-d$, since
$G/P_d\cong G/P_{n-d}$. Denote the set of minimal representatives
of $W/W_{P_d}$ by $W^{P_d}$. Then the Schubert varieties in
$G_{d,n}(\cong G/P_d)$ are indexed by $W^{P_d}$.
\subsection{The elements $w_1, w_2$}\label{elts} The unique maximal element $w_0^{P_d}$
corresponding to $G_{d,n}$ has the following reduced expression
$$w_0^{P_d}=u_1u_2\cdots u_d,\  u_k=s_{n-d+k-1}s_{n-d+k-2}\cdots
s_k$$ Let us denote $w_0^{P_d}$ by just $w_1$. Similarly,
considering the Type $\mathbf{A}_{n-1}$ Dynkin diagram with simple
roots

\ni
$\alpha_{d-1},\alpha_{d-2},\cdots,\alpha_1,\alpha_0,\alpha_{n-1},
\alpha_{n-2},\cdots,\alpha_{d+1}$ (taken in that order),

\ni namely,
$$\xymatrix{ \alpha_{d-1} \ar[r]  & \cdots \ar[r] & \alpha_1 \ar[r]
& \alpha_0  \ar[r]& \alpha_{n-1} \ar[r] &\cdots  \ar[r]&
\alpha_{d+1}}\leqno{(B)}$$ the unique maximal element $w_2$ in the
set of minimal representatives corresponding to omitting
$\alpha_0$ has the following reduced expression
$$v_{d-1}v_{d-2}\cdots v_1v_0,\  v_k=s_{d+k+1}s_{d+k+2}\cdots
s_{n-1}s_0s_1s_2\cdots s_k, 0\le k\le d-1$$ Note that for $0\le
k\le d-2$, we have, $d+k+1\le n-1$ (since $d\le n-d$). For $k=d-1$
again, $d+k+1\le n-1$, if $d<n-d$; if $d=n-d$, then
$v_{d-1}=s_0s_1s_2\cdots s_{d-1}$, and in this case,
$s_{d+k+1}=(s_{2d}=s_n)$ is to be understood as $s_0$.

In the sequel, we shall refer to the two systems as system A,
system B, respectively. We shall index system B as
$\alpha'_1,\cdots,\alpha'_{n-1}$, and denote the corresponding
reflections by $s'_1,\cdots, s'_{n-1}$; we have,
$$s'_{d-k}=s_k, 1\le k\le d-1,\ s'_{n+d-\ell}=s_l,
d+1\le \ell\le n-1$$

We define $\kappa_d=w_1w_2$. Then using the lifts for $s_i, 0\le
i\le n-1$, as described in \S \ref{gen}, it is easily checked that
$\kappa_d$ is the diagonal matrix with three diagonal blocks:
$$\kappa_d=diag([tI_{d}],[I_{n-2d}],[t^{-1}I_{d}])$$ (here, $I_r$ denotes the identity $r\times r$ matrix).

 Let $P_d$ be as above. Let
$P'_d$ be the maximal parabolic subgroup of system B,
corresponding to omitting $\alpha'_d(=\alpha_0)$.

 We now prove three properties of $\kappa_d$. In the
discussion below, we shall have the following notation:

$\bullet$ $R$ (respectively, $R^+$) is the root system
(respectively, the system of positive roots) of $G$

$\bullet$ $R_{P_{d}}$ (respectively, $R^+_{P_{d}}$) is the root
system (respectively, the system of positive roots) of ${P_{d}}$

$\bullet$ $W_{P_{d}}$ is the Weyl group of ${P_{d}}$

$\bullet$ $W^{P_{d}}$ is the set of minimal representatives of
$W/W_{P_{d}}$

$\bullet$ ${\widehat{W}}^{\mathcal{Q}}$ is the set of minimal
representatives of ${\widehat{W}}/W_{\mathcal{Q}}, {\mathcal{Q}}$
being the two-step parabolic subgroup of ${\mathcal{G}}$,
corresponding to omitting $\alpha_0,\alpha_d$.

We will also need the description of the set of
 positive real roots of
$\mathcal{G}$ (cf. \cite{kac}): the set of positive real roots is
given by

\ni $\{q\delta+\beta,q> 0, \beta\in R\}\dot\cup\{R^+\}$, where,
recall that $\delta=\alpha_0+\theta$,
$\theta=(\alpha_1+\cdots+\alpha_{n-1})$ is the highest root (in
$R^+$).

\ni We have the braid relations
$$\begin{gathered}s_is_{i+1}s_{i}=s_{i+1}s_{i}s_{i+1},
1\le i\le n-2,\\
s_0s_{1}s_{0}=s_{1}s_{0}s_{1},\
s_0s_{n-1}s_{0}=s_{n-1}s_{0}s_{n-1}\end{gathered}$$ and the
commuting relations:
$$s_is_j=s_js_i, 1\le i,j\le n-1, |i-j|>1,\ \  s_0s_i=s_is_0,
2\le i\le n-2$$

\subsection{I. A reduced expression for $\kappa_d$:}\label{Red} In this
subsection, we shall prove that the expression $\kappa_d=w_1w_2$,
with the reduced expressions for $w_1,w_2$ as in \S \ref{elts}, is reduced. We
prove the following more general Lemma.
\begin{lem}\label{red}
Let $y_1,y_2$ be in $W^{P_d}, W^{P'_d}$ respectively, and let
$y_1=s_{i_1}\cdots s_{i_\ell},y_2=s'_{j_1}\cdots s'_{j_t}$ be
reduced expressions for $y_1, y_2$. Then

\ni $s_{i_1}\cdots s_{i_\ell}s'_{j_1}\cdots s'_{j_t}$ is a reduced
expression for $y_1y_2$.
\end{lem}
\begin{proof} First observe that any reduced
expression for $y_1$ (respectively, $y_2$) is a right-end segment
of the reduced expression for $w_1$ (respectively, $w_2$) as
described in \S \ref{elts}, and hence $$s_{i_\ell}=s_d,\
s'_{j_t}=s_0 \leqno{(*)}$$ (see the descriptions of the reduced
expressions for $w_1,w_2$ as given in \S \ref{elts}).
 Let us denote the word $s_{i_1}\cdots
s_{i_\ell}s'_{j_1}\cdots s'_{j_t}$ as just $r_1\cdots r_{\ell+t}$, 
where $r_j, 1\le j\le \ell+t$ equals the reflection
$s_{\beta_j},\beta_j$ being either an $\alpha_a$ or an
$\alpha'_b$; let us denote $\ell+t$ by $m$. We shall show that
$r_1\cdots r_{j-1}(\beta_j)>0$, i.e., a positive root, for any
left-end sub word $r_1\cdots r_{j-1},2\le j\le m$, from which the
required result will follow. This is clear for $2\le j\le \ell$,
since $r_1\cdots r_{\ell}(=y_2)$ is reduced. Let us then consider
$r_1\cdots r_{j-1}(\beta_j),j\ge \ell+1$. We now divide the proof
into the following two cases:

\ni\textbf{Case 1:} Let $j=\ell+1$. Then $\beta_j=\alpha_k$, for
some $k, 0\le k\le n-1, k\ne d$. If $k>0$, then $\alpha_k\in
S_{P_d}$ (the system of simple roots of $P_d$), and hence
$r_1\cdots r_{j-1}(\beta_j)=y_1(\alpha_k)>0$, since $y_1\in
W^{P_d}$, and the result follows. If $k=0$, then $r_1\cdots
r_{j-1}(\beta_j)=y_1(\alpha_0)$ is clearly a positive root, since
$y_1$ being in $W^{P_d}$, the reduced expression for $y_1$ does
not involve $s_0$, and the result follows.

\ni\textbf{Case 2:} Let $j\ge\ell+2$. Now $r_{\ell+1}\cdots
r_{j-1}(\beta_j)>0$, a positive (real) root (since
$r_{\ell+1}\cdots r_{j}$ being a left segment of the reduced
expression $r_{\ell+1}\cdots r_{m}$ is reduced); hence,
$r_{\ell+1}\cdots r_{j-1}(\beta_j)$ is of the form
$q\delta+\gamma,q\ge 0,\gamma\in R$. We now divide the proof into
the following two subcases:

\ni\textbf{Subcase 2(a):} Let $q>0$. Then, $r_{1}\cdots
r_{j-1}(\beta_j)=r_1\cdots
r_{\ell}(q\delta+\gamma)(=y_1(q\delta+\gamma))$ is of the form
$q\delta+\epsilon$, where $\epsilon\in R$ (since $y_1$ does not
involve $s_0$), and the result follows.

\ni\textbf{Subcase 2(b):} Let $q=0$. Then $\gamma=r_{\ell+1}\cdots
r_{j-1}(\beta_j)$ is a positive root; further, it is in
$R^+_{P_d}$. Hence $r_{1}\cdots r_{j-1}(\beta_j)=r_1\cdots
r_{\ell}(\gamma)(=y_1(\gamma))>0$ (since, $y_1\in W^{P_d}$), and
the result follows.

This completes the proof of the Lemma.
\end{proof}
\subsection{II. Minimal representative property for $\kappa_d$:}\label{Min} In
this subsection, we shall prove that $\kappa_d$ is in
${\widehat{W}}^{\mathcal{Q}}$ (where recall that $\mathcal{Q}$ is
the two-step parabolic subgroup of $\mathcal{G}$ corresponding to
omitting $\alpha_0,\alpha_d$). We prove the following more general
Lemma.
\begin{lem}\label{min}
Let $y_1,y_2$ be in $W^{P_d}, W^{P'_d}$ respectively. Then
$y_1y_2$ is in ${\widehat{W}}^{\mathcal{Q}}$.
\end{lem}
\begin{proof}
We shall show that $y_1y_2(\alpha_i)>0, i\ne 0,d$, from which the
required result will follow. Let $y_2(\alpha_i)=\beta$. Then
$\beta>0$, a positive real root (since $y_2\in W^{P'_d}$), say,
$\beta=q\delta+\gamma,q\ge 0, \gamma\in R $. We now divide the
proof into the following two cases:

\ni\textbf{Case 1:} Let $q>0$. Then
$y_1y_2(\alpha_i)=y_1(q\delta+\gamma)$. Now $y_1(q\delta+\gamma)$
is of the form $q\delta+\epsilon$, where $\epsilon\in R$ (since
$y_1$ does not involve $s_0$). Hence $y_1y_2(\alpha_i)>0$.

\ni\textbf{Case 2:} Let $q=0$. Then
$y_1y_2(\alpha_i)=y_1(\gamma)$. Now $\gamma=y_2(\alpha_i)$ is a
positive root; further, it is in $R^+_{P_d}$ (since $y_2$ does not
involve $s_d$). Hence $y_1(\gamma)>0$ (since $y_1\in W^{P_d}$),
and therefore $y_1y_2(\alpha_i)(=y_1(\gamma))>0$, as required.

This completes the proof of the Lemma.
\end{proof}

Combining Lemmas \ref{red}, \ref{min}, we obtain the following

\begin{cor}\label{dim}
For the Schubert variety $X(\kappa_d)$ in
$\mathcal{G}/\mathcal{Q}$, we have, dim$\,X(\kappa_d)=2d(n-d)$ (=2
dim$\,G_{d,n}$).
\end{cor}

\subsection{III. $G_0$-stability:}\label{sysb}  As in \S \ref{elts},
we shall index system B as $\alpha'_1,\cdots,\alpha'_{n-1}$, and
denote the corresponding reflections by $s'_1,\cdots, s'_{n-1}$;
we have,
$$s'_{d-k}=s_k, 1\le k\le d-1,\ s'_{n+d-\ell}=s_l, d+1\le \ell\le n-1$$
Let $P_d,P'_d$ be as in \S \ref{Red}.
\begin{lemma}\label{stab}
For system A, we have, \begin{enumerate} \item
$s_kw_1=w_1s_{d+k},1\le k\le n-d-1$ \item
$s_{\ell}w_1=w_1s_{\ell-(n-d)},n+1-d\le \ell\le n-1$
\end{enumerate}
\end{lemma}

\ni(here, $w_1$ is as in \S \ref{elts}).
\begin{proof} As an element of $S_n$, we have,
$w_1=([n-d+1,n][1,n-d])$ (here, for a pair of integers $i<j$,
$[i,j]$ denotes the set $\{i,i+1,\cdots, j\}$; also, the notation
$w=(a_1\cdots,a_n)$ for a permutation is the usual one-line
notation, namely, $w(i)=a_i, 1\le i\le n$). Given a permutation
$w=(a_1\cdots,a_n)$, and a pair of integers $i,j, 1\le i,j\le n$
let $i=a_k, j=a_l$. Then $$s_{(i,j)}w=ws_{(k,l)}$$ (here, for
$1\le a,b\le n,a\ne b$, $s_{(a,b)}$ denotes the transposition (of
switching $a$ and $b$)). The assertions 1 and 2 follow from this
and the facts that writing $w_1=(a_1\cdots,a_n)$, we have,
$$\begin{gathered}w_1(d+k)=k, 1\le k\le n-d-1,\\
 w_1(\ell-(n-d))=\ell,
n+1-d\le \ell\le n-1 \end{gathered}$$
\end{proof}

\ni As a straight forward consequence, we have similar results for
system B:
\begin{cor}\label{sysb'}
For system B, we have, \begin{enumerate} \item
$s'_kw_2=w_2s'_{d+k},1\le k\le n-d-1$ \item
$s'_{\ell}w_2=w_2s'_{\ell-(n-d)},n+1-d\le \ell\le n-1$
\end{enumerate}
\end{cor}

 \ni(here, $w_2$ is as in \S \ref{elts}).

 Using Lemma \ref{stab}, Corollary \ref{sysb'}, \S \ref{Red},
 and \S \ref{Min},
  we shall now show the $G_0$-stability for the Schubert variety
$X_{\kappa_d}$ in
 $\mathcal{G}/\mathcal{Q}$ (for the action on the left by
 multiplication).
 \begin{lem}\label{rel}
For $1\le k\le n-1,k\ne n-d,d, s_k\kappa_d=\kappa_d s_k$.
 \end{lem}
\begin{proof}
Recall that $\kappa_d=w_1w_2$. We may suppose that $d\le n-d$ (since
$SL(n/P_d\cong SL(n))/P_{n-d}$). We divide the proof into the
following two cases.

\ni\textbf{Case 1:} Let $1\le k\le n-d-1$

\ni We have (cf. Lemma \ref{elts}, (1)),

\ni $s_k\kappa_d=s_kw_1w_2=w_1s_{d+k}w_2$. Now, in view of \S
\ref{sysb} (with $\ell =d+k$), we have that
$s_{d+k}=s'_{n+d-(d+k)}=s'_{n-k}$, and hence
$s_{d+k}w_2=s'_{n-k}w_2$. Thus we get

 $$s_k\kappa_d=w_1s'_{n-k}w_2\leqno{(*)}$$

To compute $s'_{n-k}w_2$, we further divide this case into the
following two subcases:

\ni\textbf{Subcase 1(a):} Let $k\ge d+1$.

\ni This implies $n-k\le n-d-1$. Hence, in view of Corollary
\ref{sysb'}, we get that $s'_{n-k}w_2=w_2s'_{d+n-k}$. Now,
$d+n-k\ge d+1$ (since $k\le n-1$). Hence we obtain (in view of \S
\ref{elts}), $w_2s'_{d+n-k}=w_2s_k$. Thus $s'_{n-k}w_2=w_2s_k$;
substituting in (*), we get the required result.

\ni\textbf{Subcase 1(b):} Let $k\le d-1$.

\ni This implies that $n-k\ge n-d+1$. Hence in view of Corollary
\ref{sysb'}, we get that
$s'_{n-k}w_2=w_2s'_{n-k-(n-d)}=w_2s'_{d-k}$. Now the hypothesis
that $k\le d-1$ implies (cf. \S \ref{elts})) that $w_2s'_{d-k}
=w_2s_k$. Thus $s'_{n-k}w_2=w_2s_k$; substituting in (*), we get
the required result.

\ni\textbf{Case 2:} Let $k\ge n+1-d$.

We have, by Lemma \ref{stab}, (2), $s_kw_1=w_1s_{k-(n-d)}$. Now,
$k-(n-d)\le d-1$, since $k\le n-1$. Hence, we have (cf. \S
\ref{elts}), $s_{k-(n-d)}=s'_{d-(k-(n-d))}=s'_{n-k}$. Hence
$s_{k-(n-d)}w_2=s'_{n-k}w_2=w_2s'_{n-k+d}$ (cf. Corollary
\ref{sysb'}; note that $n-k\le d-1$, since $k\ge n+1-d$). Now,
$k\ge d+1$, since $k\ge n+1-d$ and $n-d\ge d$. Hence, by (cf. \S
\ref{elts}), we get that $s'_{n-k+d}=s_k$, and therefore,
$s_{k-(n-d)}w_2=w_2s_k$; substituting in (*), we get the required
result.
\end{proof}

\begin{prop}\label{left}
The Schubert variety $X(\kappa_d)$ in $\mathcal{G}/\mathcal{Q}$,
($\mathcal{Q}$ being the two-step parabolic subgroup of
$\mathcal{G}$, corresponding to omitting $\alpha_0,\alpha_d$) is
stable for multiplication on the left by $G_0$.
\end{prop}
\begin{proof}
We need to show that
$s_k\kappa_d\le\kappa_d(mod\,\mathcal{Q}),1\le k\le n-1$. For
$k=n-d$, this is clear, since in the reduced expression
(Proposition \ref{red}) for $\kappa_d=w_1w_2$, we have that the
reduced expression for $w_1$ starts with $s_{n-d}$ (cf. \S
\ref{elts}).

For $1\le k\le n-1,k\ne d, n-d$, we have, by Lemma \ref{rel},
$s_k\kappa_d=\kappa_d s_k$, and the result follows (since, for
$1\le k\le n-1,k\ne d, s_k\in W_{\mathcal{Q}}$).

Let now $k=d$. If $d=n-d$, then as above, we have that the reduced
expression for $w_1$ starts with $s_{d}$ (cf. \S \ref{elts}) and
the result follows. Let then $d\le n-d-1$. We have, by Lemma
\ref{stab}, (1), $s_dw_1=w_1s_{2d}$. Now

\ni $s_{2d}w_2=s_{2d}v_{d-1}v_{d-2}\cdots v_1v_0$,

\ni where $v_k,0\le k\le d-1$ has the reduced expression

$v_k=s_{d+k+1}s_{d+k+2}\cdots s_{n-1}s_0s_1s_2\cdots s_k$ (cf. \S
\ref{elts}). In particular, $v_{d-1}$ begins with $s_{2d}$. Hence,
we obtain $s_dw_1w_2=w_1s_{2d}w_2<w_1w_2$, and the result follows.
\end{proof}

%%%%%%%%%%%%%%%%%%%%%%%%%%%%%%%%%%%%%%%%%%%%%%
\section{The main Lemma}\label{main} In this section, we prove one
crucial result involving $\kappa_d$ which will be used for proving
the main result (namely, Theorem \ref{ctgt}).
\begin{lemma}\label{crucial}
Let $Y=\sum_{1\le i\le d<j\le n}\, a_{ij}E_{ij}, E_{ij} $ being
the elementary $n\times n$ matrix with $1$ at the $(i,j)$-th place
and 0's elsewhere. Let ${\underline{Y}}=Id_{n\times n}+\sum_{1\le
i\le n-1}\,t^{-i}Y^i$ (note that $Y^n=0$).  There exist $g\in G_0,
h\in\mathcal{Q}$ such that $g\kappa_d = {\underline{Y}}h $ (recall
that $\mathcal{Q}$ is the two-step parabolic subgroup of
${\widehat{SL_n}}$, corresponding to omitting
$\alpha_0,\alpha_d$).
\end{lemma}
\begin{proof}
We shall now show that a choice of $g_{ij}, 1\le i,j\le n$, and
$h_{ij}, 1\le i,j\le n $ exists so that $h\in\mathcal{Q}$, and
$g\kappa_d = {\underline{Y}}h $. We have, ${\underline{Y}}^{-1}=
Id_{n}-t^{-1}Y$.  Set
$$h=(Id_{n}-t^{-1}{\underline{Y}})g\kappa_d\leqno{(*)}$$ We have
(by definition of $\kappa_d$ (\S\ref{elts}))
$$\kappa_d=diag([tId_{d}],[Id_{n-2d}],[t^{-1}Id_{d}])$$
Note that since we want $h$ to belong to $\mathcal{Q}$, the
condition on $h$ is that $h(0)$ should belong to $P_d$. Hence,
$h_{ij}$'s should satisfy the following conditions:

\ni\textbf{Condition 1:} $h_{ij},j\le d<i\le n$ should have order
$>0$ (as an element of $K[[t]]$)

\ni\textbf{Condition 2:} The remaining $h_{ij}$'s should have
order $\ge 0$. Now using (*), we describe below the columns of
$h$:
$$
\begin{pmatrix}
h_{1j}\\
\vdots\\
h_{dj}\\
h_{d+1j}\\
\vdots\\
h_{nj}
\end{pmatrix}=
\begin{pmatrix}
tg_{1j}-\sum_{d+1\le i\le n}a_{1i}g_{ij}\\
 \vdots\\
 tg_{dj}-\sum_{d+1\le i\le n}a_{di}g_{ij}\\
 tg_{d+1j}\\
 \vdots\\
 tg_{nj}
\end{pmatrix},\ 1\le j\le d\leqno{(I)}$$
$$\begin{pmatrix}
h_{1j}\\
\vdots\\
h_{dj}\\
h_{d+1j}\\
\vdots\\
h_{nj}
\end{pmatrix}=
\begin{pmatrix}
g_{1j}-\sum_{d+1\le i\le n}t^{-1}a_{1i}g_{ij}\\
 \vdots\\
 g_{dj}-\sum_{d+1\le i\le n}t^{-1}a_{di}g_{ij}\\
 g_{d+1j}\\
 \vdots\\
 g_{nj}
\end{pmatrix},\ d+1\le j\le n-d\leqno{(II)}$$

\vskip.5cm
$$\begin{pmatrix}
h_{1j}\\
\vdots\\
h_{dj}\\
h_{d+1j}\\
\vdots\\
h_{nj}
\end{pmatrix}=
\begin{pmatrix}
t^{-1}g_{1j}-\sum_{d+1\le i\le n}t^{-2}a_{1i}g_{ij}\\
 \vdots\\
 t^{-1}g_{dj}-\sum_{d+1\le i\le n}t^{-2}a_{di}g_{ij}\\
 t^{-1}g_{d+1j}\\
 \vdots\\
 t^{-1}g_{nj}
\end{pmatrix},\ n+1-d\le j\le n\leqno{(III)}
$$ Condition 1 follows from (I). Also, from (I) we get that Condition
2 holds for $h_{ij}, 1\le i,j\le d$. Thus the entries in the first
$d$ columns of $h$ satisfy the required conditions.

\ni From (II), we get that Condition 2 holds for $h_{ij}, d+1\le
i\le n, d+1\le j\le n-d$. Regarding the entries $h_{ij}, 1\le i\le
d, d+1\le j\le n-d$, we shall choose $g_{ij}, d+1\le i\le n,
d+1\le j\le n-d$ so that order$\,g_{ij}>0$. Thus with this choice,
we have that the entries in the  $j$-th column, $d+1\le j\le n-d$,
of $h$ satisfy the required conditions.

\ni In view of (III), we shall choose $g_{ij}, d+1\le i\le n,
n-d+1\le j\le n$ so that order$\,g_{ij}=1$. With this choice, we
obtain that $h_{ij},d+1\le i\le n, n-d+1\le j\le n$ satisfy
Condition 2. In order to have $h_{ij},1\le i\le d, n-d+1\le j\le
n$ satisfy Condition 2, we choose order$\,g_{ij}=0,1\le i\le d,
n-d+1\le j\le n$, and impose the following conditions. We write
$g_{ij}=\sum\,g_{ij}^{(r)}t^r$. Then the conditions are
$$g_{ij}^{(0)}-\sum_{d+1\le m\le n}a_{im}g_{mj}^{(1)}=0,1\le i\le d,
n-d+1\le j\le n\leqno{(**)}$$ Treating $g_{ij}^{(0)},1\le i\le
d,g_{mj}^{(1)},d+1\le m\le n,n-d+1\le j\le n$, (**) is a linear
homogeneous system of $d^2$ equations in $nd$ variables, and hence
there exist non-trivial solutions (note that $nd>d^2$, since $d\le
n-1$). Hence, we can choose $g_{ij},1\le i\le d,g_{mj},d+1\le m\le
n,n-d+1\le j\le n$ so that all of the entries in $j$-th column,
$n-d+1\le j\le n$ satisfy Condition 2.

This completes the proof of the Lemma.
\end{proof}
%%%%%%%%%%%%%%%%%%%%%%%%%%%%%%%%%%%%%%%%%%%%%%%%%%%%%%%%%%%
\section{Cotangent bundle}\label{bundle} The cotangent bundle
$T^*G/P_d$ is the vector bundle over $G/P_d$, the fiber at any
point $x\in G/P_d$ being the cotangent space to $G/P_d$ at $x$;
the dimension of $T^*G/P_d$ equals $2\,$dim\,$G/P_d=2d(n-d)$.
Also, $T^*G/P_d$ is the fiber bundle over $G/P_d$ associated to
the principal $P_d$-bundle $G\rightarrow G/P_d$, for the Adjoint
action of $P_d$ on $\,u({P_d})$ (=Lie($\,U({P_d})$), Lie algebra
of $U({P_d})$, $U({P_d})$ being the unipotent radical of $P_d$)).
Thus
$$T^*G/P_d=G\times^{P_d} u({P_d})=G\times u({P_d})/
\sim$$ where the equivalence relation $\sim$ is given by $(g,
Y)\sim(gx, x^{-1}Yx),g\in G,Y\in u({P_d}), x\in P_d$ .

\subsection{Embedding of $T^*G/P_d$ inside
$\mathcal{G}/\mathcal{Q}$} Define $\phi:G\times^{P_d}
u(P_d)\rightarrow \mathcal{G}/\mathcal{Q}$ as
$$\phi(g,Y)=g(Id+t^{-1}Y+t^{-2}Y^2+\cdots)(mod\,\mathcal{Q}),
g\in G, Y\in{u(P_d)}$$ Note that the sum on the right hand side is
finite (since $Y$ is nilpotent). In the sequel, we shall denote
$${\underline{Y}}:=Id+t^{-1}Y+t^{-2}Y^2+\cdots$$ We shall now list
some facts on the map $\phi$:

\ni\textbf{(i) $\phi$ is well-defined:} Let $g\in G,x\in P_d, Y\in
u(P_d)$. We have,

 \ni $ \phi((gx, x^{-1}Yx))$

\ni $=gx(Id+t^{-1}x^{-1}Yx+t^{-2}x^{-1}Y^2x+\cdots)
(mod\,\mathcal{Q})$

\ni $ =g(x+t^{-1}Yx+t^{-2}Y^2x+\cdots) (mod\,\mathcal{Q})$

\ni $\equiv g(Id+t^{-1}Y+t^{-2}Y^2+\cdots)(mod\,\mathcal{Q})$

\ni $ =\phi(g,Y) $

\ni\textbf{(ii) $\phi$ is injective:} Let
$\phi((g_1,Y_1))=\phi((g_2,Y_2))$. This implies that
$g_1{\underline{Y_1}}\equiv
g_2{\underline{Y_2}}(mod\,\mathcal{Q})$, where recall that for
$Y\in{u(P_d)},{\underline{Y}}=Id+t^{-1}Y+t^{-2}Y^2+\cdots$. Hence,
$g_1{\underline{Y_1}}= g_2{\underline{Y_2}}x$, for some
$x\in\mathcal{Q}$. Denoting $h=:g_2^{-1}g_1$, we have,
$h{\underline{Y_1}}={\underline{Y_2}}x$, and
therefore,$$x={\underline{Y_2}}^{-1}h{\underline{Y_1}}=
{\underline{Y_2}}^{-1}(h{\underline{Y_1}}
h^{-1})h={\underline{Y_2}}^{-1}{\underline{Y'_1}}h$$ where
${\underline{Y'_1}}=h{\underline{Y_1}} h^{-1}$. Hence
$$xh^{-1}={\underline{Y_2}}^{-1}{\underline{Y'_1}}=(Id-t^{-1}Y_2)
(Id+t^{-1}hY_1h^{-1}+t^{-2}hY^2_1h^{-1}+\cdots)$$ Now, left hand
side is integral (since, $x\in\mathcal{Q},h(=g_2^{-1}g_1)\in G$).
Hence both sides equal $Id$. This implies
$${\underline{Y_2}}={\underline{Y'_1}},\ x=h$$ The fact that $x=h$
together with the facts that $x\in\mathcal{Q},h\in G$ implies that
$$h\in \mathcal{Q}\cap G(=P_d)\leqno{(*)}$$ Further, the fact that
${\underline{Y_2}}={\underline{Y'_1}}$ implies that
${\underline{Y_1}}=h^{-1}{\underline{Y_2}}h$ (note from above that
${\underline{Y'_1}}=h{\underline{Y_1}} h^{-1}$). Hence
$$Id+t^{-1}Y_1+t^{-2}Y_1^2+\cdots=Id+t^{-1}h^{-1}Y_2h+t^{-2}h^{-1}Y^2_2h
+\cdots$$ From this it follows that $$Y_1=h^{-1}Y_2h\leqno{(**)}$$
Now (*), (**) together with the fact that $h=g_2^{-1}g_1$ imply
that $$(g_1,Y_1)=(g_2h,h^{-1}Y_2h)\sim (g_2,Y_2)$$ From this
injectivity of $\phi$ follows.

\ni\textbf{(iii) $G$-equivariance:} $\phi$ is $G$-equivariant
(clearly).

\begin{thm}\label{ctgt} The map $\theta$ identifies
${\overline T^*G/P_d}$ ( the closure being in
$\mathcal{G}/\mathcal{Q}$) with the affine Schubert variety
$X(\kappa_d)$.
\end{thm}
\begin{proof} Let $(g_0,Y), g_0\in G, Y\in{u(P_d)}$. Then $Y$ is
of the form: $$Y=\sum_{1\le i\le d<j\le n}\, a_{ij}E_{ij}, E_{ij}
$$ Now $\theta(g_0,
Y)=g_0(Id+t^{-1}Y+t^{-2}Y^2+\cdots)(mod\,\mathcal{Q})=
g_0{\underline{Y}}(mod\,\mathcal{Q})$, where
${\underline{Y}}=Id+t^{-1}Y+t^{-2}Y^2+\cdots$. Then Lemma
\ref{crucial} implies that there exist $g\in G_0, h\in\mathcal{Q}$
such that $g\kappa_d = {\underline{Y}}h $. Hence ${\underline{Y}}$
belongs to $X(\kappa_d)$; hence $g_0{\underline{Y}}$ is also in
$X(\kappa_d)$ (since $g_0$ is clearly in $G_0$). Hence
$T^*G/P_d\subset X(\kappa_d)$, and therefore
${\overline{T^*G/P_d}}\subseteq X(\kappa_d)$. Now by dimension
considerations, we obtain that ${\overline{T^*G/P_d}}=
X(\kappa_d)$ (note (cf. Corollary \ref{dim}),
dim$\,X(\kappa_d)=2d(n-d)$)
\end{proof}

\end{document}